\documentclass[11pt]{amsart}
\usepackage{amsmath, amsthm, amssymb,amsfonts}
\usepackage{verbatim,geometry,graphicx}
\usepackage{enumerate,matlab-prettifier}
\usepackage{algpseudocode,algorithm}

\usepackage{datetime}

\theoremstyle{plain}
\newtheorem{thm}{Theorem}[section]

\newtheorem{example}[thm]{Example}

\newtheorem{lemma}[thm]{Lemma}
\newtheorem{coro}[thm]{Corollary}

\newtheorem{defn}[thm]{Definition}

\newtheorem{rem}[thm]{Remark}
\newtheorem{rems}[thm]{Remarks}

\newcommand{\R}{{\mathbb R}}

\newcommand{\N}{{\mathbb N}}
\newcommand{\NxN}{{N\times N}}
\newcommand{\diag}{\operatorname{diag}}

\newcommand{\bmat}[1]{\begin{bmatrix}#1\end{bmatrix}}
\newcommand{\abs}[1]{\vert #1 \vert}
\newcommand{\lam}{\lambda}

\newcommand{\cS}{\mathcal{S}}

\newcommand{\rank}{\operatorname{rank}}

\newcommand{\smat}[1]{ \left[\begin{smallmatrix} #1 \end{smallmatrix}\right]}

\def\svdots{\vbox{\baselineskip=1.5pt\lineskiplimit=0pt
		\kern1.5pt \hbox{$\ss .$}\hbox{$\ss .$}\hbox{$\ss .$}}}

\def\sddots{\mathinner{\raise3pt\vbox{\hbox{$\ss .$}}
		\raise1.5pt\hbox{$\ss .$}\hbox{$\ss .$}}}

\renewcommand{\ss}{\scriptstyle}
\newcommand{\Rn}{\R^n}

\newcommand{\pder}[2]{\frac{\partial #1}{\partial #2}}

\renewcommand{\vec}[1]{{\boldsymbol{#1}}}
\def \x {{\vec{x}}}
\def \z {{\vec{z}}}
\def \y {{\vec{y}}}
\def \v {{\vec{v}}}
\def \w {{\vec{w}}}
\def \e {{\vec{e}}}
\def \b {{\vec{b}}}
\def \c {{\vec{c}}}
\def \f {{\vec{f}}}

\def \q {{\vec{q}}}
\def \u {{\vec{u}}}

\def \F {{\vec{F}}}

\title{On an inverse tridiagonal eigenvalue problem and \\ its application to synchronization of network motion}

\author{Luca Dieci}
\address{School of Mathematics \\ Georgia Tech \\
	Atlanta, GA 30332 U.S.A.}
\email{dieci@math.gatech.edu}
\author{Cinzia Elia}
\address{Dipartimento di Matematica, Univ. of Bari, I-70100, Bari,
	Italy}
\email{cinzia.elia@uniba.it}
\author{Alessandro Pugliese}
\address{Dipartimento di Matematica, Univ. of Bari, I-70100, Bari,
	Italy}
\email{alessandro.pugliese@uniba.it}
\thanks{
The authors gratefully acknowledge the support of the School of Mathematics of 
Georgia Tech.  The work has been partially supported by the 
GNCS-Indam group and the PRIN2022PNR P2022M7JZW 
`` SAFER MESH- Sustainable mAnagement oF watEr Resources modEls and numerical MetHods '' research grant. }
\subjclass{65F18, 15A18}

\keywords{Tridiagonal network, inverse eigenvalue problem, synchronization}

\begin{document}

%
%
%
\begin{abstract}
In this work, motivated by the study of stability of the synchronous orbit of a network with tridiagonal Laplacian matrix, we first solve an inverse eigenvalue problem which builds a tridiagonal Laplacian matrix with eigenvalues $\lambda_1=0<\lambda_2<\cdots <\lambda_N$ and null-vector $\e=\smat{1\\ \svdots \\ 1}$.  Then, we show how this result can be used to guarantee --if possible-- that a synchronous orbit of a connected tridiagonal network associated to the matrix $L$ above is asymptotically stable, in the sense of having an associated negative Master Stability Function (MSF).  We further show that there are limitations when we also impose symmetry for $L$.
\end{abstract}
\maketitle
\pagestyle{myheadings}
\thispagestyle{plain}
\markboth{L. Dieci, C. Elia, A. Pugliese}{Inverse tridiagonal problem and network stabilization}

\medskip
\noindent {\bf Notation}. 
We let $\e_k$ be the $k$-th column of the identity matrix, and $\e$ be the vector of
all $1$'s.  Boldface will indicate vectors, whose number of elements will be clear from the context.  As needed, we will use {\tt Matlab} notation for entries of vectors and matrices, e.g. $A(i,j)$ will be the $(i,j)$-th entry of a matrix $A$.

\smallskip

\section{A linear algebra and a dynamical systems problem}
The purpose of this work is two-fold.  First, we will solve the following inverse eigenvalue problem: ``{\emph{Given values $\lambda_1=0<\lambda_2<\cdots <\lambda_N$, find an unreduced tridiagonal matrix $L$ with these as eigenvalues, with positive diagonal and negative off diagonal entries, and whose entries on each row add up to $0$ (that is, $L\e=0$)}}''.  Then, we will use this result to: ``{\emph{Show how and when we can guarantee that a synchronous orbit of a connected tridiagonal network associated to the matrix $L$ above is asymptotically stable, in the sense that the Master Stability Function (MSF) is negative}}''. 

In the process of solving the above problems, we will also realize that it is generally not possible to also require that $L$ is symmetric.  Further, we will show that if the synchronous orbit of the network associated to the standard tridiagonal matrix $T$ in \eqref{TridStandard}
does not lead to a negative value of the MSF, then there cannot be any connected network with tridiagonal symmetric Laplacian matrix giving a negative MSF.  

From the theoretical point of view, our results are an important addition to the vast literature on inverse eigenvalue problems for tridiagonal matrices, by adding the nontrivial constraint of having a specific null vector.  From the application point of view, our results provide a rigorous and clear indication of how to weigh the arcs of a network with the nearest neighbor topology in order to obtain a stable synchronous orbit, thereby overcoming the limitations imposed by requiring symmetry.

A plan of the paper is as follows.  In the remaining part of this section, we review some relevant linear algebra and network (dynamical systems) results, in particular about the concept of MSF.  Then, in Section \ref{LinAlg}, we give our main theoretical results about the inverse eigenvalue problem, by giving and rigorously justifying
novel algorithms to obtain a tridiagonal matrix with given spectrum and null vector.  
In Section \ref{TridMSF}, we look at some theoretical results about symmetric connected tridiagonal network structure, and in Section \ref{NumRes} we show
performance of our techniques.   Finally, conclusions are in Section \ref{Concl}.

\subsection{Linear algebra background}

\begin{defn}\label{Reducible}
Given a matrix $A\in \R^{N\times N}$, we say that $A$ is {\emph{reducible}} if there exists a permutation $P$ and integers $N_1\ne 0 \ne N_2$, $N_1+N_2=n$, such that $P^TAP=\smat{A_{11} & 0 \\ 0 & A_{22}}$ where $A_{11}\in \R^{N_1\times N_1}$, $A_{22} \in \R^{N_2\times N_2}$.  If a matrix $A$ is not reducible, then it is called {\emph{irreducible}}.
\end{defn}

\begin{defn}\label{DefTrid}
A matrix $L\in \R^{N\times N}$ is tridiagonal and unreduced if it has the form
\begin{equation}\label{GivenTrid}
	L=\bmat{a_1 & b_2 & 0 & 0 & \cdots & 0 \\
		c_2 & a_2 & b_3 & 0 & \cdots & 0 \\	
		0 & c_3 & a_3  & b_4 & \ddots & 0 \\
		\vdots &  \ddots & \ddots & \ddots & \ddots & \vdots \\
		0 & \cdots & 0 & c_{N-1} & a_{N-1} & b_N \\
		0 & \cdots & \cdots &0  & c_N & a_N}\ ,
\end{equation}
with $b_j\ne 0$, $c_j\ne 0$, $j=2,\dots, N$.  Moreover, we call $L$ {\emph{network-tridiagonal}} if 
$b_j<0$, $c_j<0$, $j=2,\dots, N$, and $a_j=-b_{j+1}-c_j$, $j=1,\dots, N$ ($c_1=0=b_{N+1}$), in which case $L$ is singular with null vector $\e$.
\end{defn}

\begin{rem}
We note that a network tridiagonal matrix $L$ is effectively the Laplacian matrix of a given connected network.  Moreover, we also remark that an unreduced tridiagonal matrix $L$ is also automatically irreducible.
\end{rem}

The following well known result (e.g., see \cite{Muir, HJ-MatAna}) will be used below.

\begin{lemma}\label{SignsOff}
Given a general tridiagonal, unreduced, matrix $L$ as in \eqref{GivenTrid},
then its eigenvalues are real and distinct, hence $L$ is diagonalizable by a real matrix of eigenvectors $V$: $V^{-1}LV=\diag(\lambda_i, \ i=1,\dots, N)$.
Moreover, 
the eigenvalues do not change as long as the products $b_kc_k$, $k=2,\dots, N$, do not	change either.
\end{lemma}
\begin{proof}
The proof follows from the fact that the characteristic polynomial of \eqref{GivenTrid} can be recursively defined as follows (Sturm sequence):
\begin{equation}\begin{split}\label{Sturm}
		& p_0(\lambda)=1\ ,\quad p_1(\lambda)=\lambda-a_1 \ , \\
		& p_j(\lambda)=(\lambda -a_j)p_{j-1}(\lambda)-(b_jc_j)p_{j-2}(\lambda) \ , \quad
		j=2,\dots, N\ . 
\end{split} \end{equation}
Note that we must have $b_jc_j\ne 0$, for all $j=2,\dots, N$, since $L$ is unreduced.
\end{proof}

\begin{example}\label{DiffCouple}
	The most commonly studied instance of network tridiagonal matrix $L$ of interest to us is that 
	associated to symmetric nearest-neighbor coupling, called 
	\emph{diffusive coupling} in \cite{Hale}.  That is, one has
	\begin{equation}\label{TridStandard}
		L=\sigma T, \quad T=\bmat{1 & -1 & 0 & 0 & \cdots & 0 \\
			-1 & 2 & -1 & 0 & \cdots & 0 \\	
			\vdots &  \ddots & \ddots & \ddots & \ddots & \vdots \\
			0 & \cdots & 0 & -1 & 2 & -1 \\
			0 & \cdots & \cdots &0  & -1 & 1}\ , 
	\end{equation}	
	where $\sigma>0$ is called the coupling strength.  The eigenvalues $\mu_j$'s and (orthonormal) eigenvectors $v_j$'s of $T$ are well known:
\begin{equation}\label{DiffCoupSpec}\begin{split}
	 \text{For}&\quad j=1,2,\dots, N: \quad \mu_j  = 4\sin^2\left(\frac{(j-1)\pi}{2N}\right) \ , \quad \text{and}\\
	v_{i1}&=\sqrt{\frac{1}{N}} \ ,\,\ v_{ij}=\sqrt{\frac{2}{N}} \cos\left(\frac{(i-0.5)(j-1)\pi}{N}\right)\ ,\,\ i =1,\dots, N \ .
\end{split}\end{equation}
%
%
%
\end{example}

\subsection{Networks and the MSF}
Consider the following system of coupled identical differential equations
\begin{equation}\label{ManyAgents}
\dot \x_i = \f(\x_i)+ \sum_{j=1}^N a_{ij} E (\x_j-\x_i),\,\ i=1,\dots, N ,	
\end{equation}
where $A=(a_{ij})_{i,j=1:N}$,  is the matrix describing the interaction of 
$N$ different agents $\x_i\in \Rn$, $a_{ii}=0, i=1,\dots,N$, $a_{ij}\ge 0$ for $j\ne i$, 
and $E\in \R^{n\times n}$ is the matrix
describing which components of each agent interact with one another.
In general, $A$ represents the structure of a directed graph with weighted edges,
and we will henceforth assume that the graph is connected, hence one can get to
any node starting from any other node, moving along edges (in particular, no row of 
$A$ can be $0$).  We recall that a graph being connected is equivalent to the matrix $A$ being irreducible.

In \eqref{ManyAgents}, separately each agent satisfies the same differential equation
\begin{equation}\label{SingleAgent}
	\dot \x = \f(\x)
\end{equation}
and by the structure of the system 
obviously taking $N$ copies of the same solution of 
\eqref{SingleAgent},
$\x_1=\x_2=\dots=\x_N=\x$,  one obtains a solution of \eqref{ManyAgents}.
This is called \emph{synchronous solution} and we denote it as $\x_s$.  Further, 
with 
\begin{equation}\label{SynchSub}
	\Xi=\{\x\in \R^{nN}:\ \x_1=\x_2=\dots =\x_N\ ,\,\x_j\in \Rn, \ j=1,\dots, N\}
\end{equation}
we indicate the {\emph{synchronous subspace}}.

It is convenient to rewrite \eqref{ManyAgents} by defining the new matrix $L$:
$L_{ij}=-a_{ij}$, for $i\ne j$ and $L_{ii}=\sum_j a_{ij}$,
$i,j=1,\dots, N$.  Obviously, $0$ is an eigenvalue of $L$, and since the
graph is connected, $0$ is a simple eigenvalue of $L$.  In general, $L$ is
not symmetric, and it corresponds to what is known as out-degree Laplacian
(see \cite{Veerman}).  In this work, we will restrict to $L$ satisfying the structure of network tridiagonal matrix given in Definition \ref{DefTrid}.

Let $\x=\begin{bmatrix} \x_1 \\ \vdots \\ \x_N \end{bmatrix}\in \R^{nN}$, 
$\F(\x)=\begin{bmatrix} \f(\x_1) \\ \vdots \\ \f(\x_N) \end{bmatrix}$, 
and rewrite \eqref{ManyAgents} as
\begin{equation}\label{ManyAgents2}
	\dot \x = \F(\x)+M \x ,\quad \text{where} \quad 
	M=-L\otimes E\ .	
\end{equation}
In order to ascertain the stability of a synchronous orbit $\x_s$ one needs to study
the   behavior of solutions of \eqref{ManyAgents2} transversal to $\Xi$ and
the {\emph{Master Stability Function}} (MSF) does precisely that.  In fact, 
the MSF tool (originally devised in \cite{PecoraCarroll}) is a 
widely adopted indicator of linearized stability of the synchronous orbit for the system
\eqref{ManyAgents2}.  The power of the technique consists in the replacement of the
large $nN$-dimensional linear system arising from linearizing \eqref{ManyAgents2},
with a single $n$-dimensional parametrized linear system.  Indeed, linearization of
\eqref{ManyAgents2} about the synchronous solution $\x_S$ gives the linear system
$$\dot \y = \bmat{D\f(\x) & & & \\ & D\f(\x) & & \\ & & \ddots & \\ & & & D\f(\x)} \y - (L\otimes E) \y,$$
where $\y=\smat{\y_1 \\ \svdots \\ \y_N}$.  Next, let $V$ be the matrix of
eigenvectors of $L$, and perform
the change of variable $(V^{-1}\otimes I_n) \y \to \y$, to obtain the $N$ 
linear systems of dimension $n$
$$\dot \y_i = D\f(\x) \y_i -\lambda_i E \y_i\ ,$$
where $\lambda_1=0<\lambda_2<\dots \lambda_N$ are the eigenvalues of $L$.
As a consequence, one considers the single parametrized linear system
\begin{equation}\label{ParamLinear}
\dot \z = (A(t) -\eta E)\z\ ,\quad \text{where} \quad \eta\ge 0 \quad \text{and} \quad 
A(t)=D\f(\x(t))\ .
\end{equation}

\begin{defn}\label{DefMSF}
The MSF is defined as the largest Lyapunov exponent (Floquet exponent,
in case the synchronous solution is a periodic orbit) of \eqref{ParamLinear} as $\eta$
ranges over the eigenvalues of $L$.  A negative value of the MSF
implies stability of the synchronous orbit.
\end{defn}

In this work, we will study directly the parametrized linear system \eqref{ParamLinear}
and a-priori decide what range of values of $\eta$ (if any) will give a negative MSF, 
then ask whether or not it is possible to find a network tridiagonal structure (that is, a Laplacian matrix of the form in \eqref{GivenTrid}) whose 
eigenvalues fit the stability region inferred by the MSF.  
More in details, if we find that there is a negative value of the MSF for $\eta \in [\eta_1, \eta_2]$, then the  synchronous solution is asymptotically stable for \eqref{ManyAgents2} if 
\begin{equation}\label{l2lN}
\frac{\lambda_2}{\lambda_N}>
\frac{\eta_1}{\eta_2}.
\end{equation}

\begin{rem}\label{StabMSF}
Obviously, the MSF depends on $E$ and it is easy to give examples where
the MSF is negative for some $E$, but positive for some other coupling matrices $E$;
for example, see \cite[Figure 1]{HuangChenLaiPecora}.
Indeed, there is a very vast literature on network synchronization, the MSF, and the interplay
between network topology and negative MSF.  This is fairly evident already at a 
graph theoretic level: e.g., by choosing the network structure of a complete graph,
and the matrix $E$ to be the identity, surely increases synchronizability.  But this
is not a desirable way to proceed, since the type of connections between agents
is not just a mathematical artifact.  For this reason, we will assume that the nearest neighbor interaction (that is, the network tridiagonal form of $L$), and the matrix $E$, are specified by the underlying application and that there is no freedom in changing them, in particular the structure of $L$.  Our goal will be to choose, if possible, the entries of $L$ so that the MSF is negative.
\end{rem}



\section{Our inverse network tridiagonal eigenvalue problem}\label{LinAlg}
We will solve our inverse eigenvalue problem in two steps. First, in 
Section \ref{SymmTridAlgo}, we give an algorithm that builds a symmetric, unreduced,
tridiagonal matrix with a given set of distinct eigenvalues.  Then, in Section 
\ref{ZeroSum}, we propose an algorithm that takes a given 
symmetric, singular, unreduced, 
tridiagonal matrix, and it produces an unreduced, generally non-symmetric,
tridiagonal matrix with a specified null vector.

There exists an extensive literature in the general area of inverse eigenvalue problems,
(e.g., see \cite{ChuGolub}).  In fact, our first algorithm (to build a symmetric tridiagonal unreduced matrix with a 
preassigned spectrum) is effectively a known result.  However, to the
best of our knowledge, our result on specifying spectrum and null vector appears
to be new, and it is what we need for obtaining a network leading to a negative
MSF.

\subsection{Symmetric unreduced tridiagonal with given 
spectrum}\label{SymmTridAlgo}
We are interested in solving the following inverse problem. \\
{\emph{Given $N$ real values
$\lambda_1 < \lambda_2 < \cdots < \lambda_N$, find an unreduced,
symmetric, tridiagonal matrix $S$, with negative off diagonal, having these
$\lambda_i$'s as eigenvalues}}.  
To clarify, we are seeking $S$ of the form
\begin{equation}\label{SymmTrid}
	S=\bmat{a_1 & b_2 & 0 & 0 & \cdots & 0 \\
		b_2 & a_2 & b_3 & 0 & \cdots & 0 \\	
		0 & b_3 & a_3  & b_4 & \ddots & 0 \\
		\vdots &  \ddots & \ddots & \ddots & \ddots & \vdots \\
		0 & \cdots & 0 & b_{N-1} & a_{N-1} & b_N \\
		0 & \cdots & \cdots &0  & b_N & a_N}\ ,
\end{equation}
with all $b_i<0$, $i=2,\dots, N$, and eigenvalues $\lambda_1 < \lambda_2 < \cdots < \lambda_N$.

To solve this problem we adopted Algorithm {\tt diag2trid} below.

\begin{algorithm}[H]
 \caption{Algorithm {\tt diag2trid}}\label{algo:diag2trid}
\begin{algorithmic}[1]
\State Set $D=\diag(\lam_1,\ldots,\lam_N)$.
\State Let $\q\in \R^N$ be the vector $\q=\e/\sqrt{N}$.
\State Let $Q$ be a Householder reflection such that $Q\e_1=\q$.
\State Set $A=Q^TDQ$.
\State Perform the \emph{Householder tridiagonalization} algorithm on $A$, so that
$$H^TAH= S \ , $$
where $S$ is tridiagonal and $H$ is orthogonal.
\State By possibly changing the signs of the $b_i$'s, 
we further enforce that $b_i<0$, $i=2,\dots, N$, to obtain the sought form of $S$.
\end{algorithmic}
\end{algorithm}

\begin{rems}\label{RemsAlgo1}  $\,$
\begin{itemize}
\item[(i)]
With the trivial exception of the last step,
Algorithm {\tt diag2trid} is effectively the same as the one we described, and justified,
in \cite{DiPuArxiv}, where it is also shown to be equivalent to a technique of Schmeisser, see \cite{Schmeisser}, whose
interest was in building an unreduced symmetric tridiagonal matrix whose characteristic polynomial is given.  
A related construction is also summarized in \cite[Theorem 4.7]{ChuGolub:Acta}.  We note that performing a possible change of sign of the $b_i$'s is legitimate
in light of Lemma \ref{SignsOff}.
\item[(ii)]
Although the choice $\q=\e/\sqrt{N}$ is our default choice, and the one we adopted in the
experiments in Section \ref{NumRes}, there is 
freedom in choosing the unit vector $\q$ in Algorithm {\tt diag2trid}.  This is natural, since the transformation $H$ is such that $H(:,1)=H(1,:)^T=\e_1$, which means 
that $S$ is diagonalized by $(QH)^T$, whose first row is given by $\q$.
Appealing to the Implicit $Q$ Theorem,
see \cite[Theorem 8.3.2]{GVL2013matrixcomp}, we know that a real 
symmetric tridiagonal matrix is completely characterized by its (real) eigenvalues 
and by the first row of the (orthogonal) matrix of its eigenvectors, in the sense that
any two real symmetric, tridiagonal, and unreduced matrices, $S$ and $T$, that are 
diagonalized by two orthogonal matrices having the same first row, must be
equal up to the sign of their off diagonal entries.
\end{itemize}
\end{rems}

The next Lemma will be needed below and it clarifies the structure of $S$ produced by Algorithm \ref{algo:diag2trid} when the eigenvalues are $\lambda_1=0 < \lambda_2 < \cdots < \lambda_N$.

\begin{lemma}\label{ai_positive}
Let $S$ as in \eqref{SymmTrid} be produced by Algorithm \ref{algo:diag2trid} relative to eigenvalues $\lambda_1=0 < \lambda_2 < \cdots < \lambda_N$.  Then, the values
$a_i$ in $S$ are all positive: $a_i>0$, $i=1,\dots, N$. 
\end{lemma}
\begin{proof}
By construction, we have $S=H^TAH=(QH)^TD(QH)$.  Let $U=QH$, so that $S=U^TDU$.  Therefore, for any $j=1,\dots, N$:
$$a_j=\e_j^TS\e_j=(U\e_j)^TD (U\e_j)=\sum_{k=2}^N\lambda_k(u_{k,j})^2$$
and so --since $\lambda_k>0$ for $k=2,\dots, N$-- we have $a_j>0$ unless $u_{kj}=0$ for $k=2,\dots, N$.  By contradiction, suppose that $u_{kj}=0$ for $k=2,\dots, N$, and some $j=1,\dots, N$.  But then $U\e_j=\pm \e_1$, and since $U$ is orthogonal this means that $\e_1^TU=\pm \e_j^T$.  But, see Remark 4-(ii), $S$ is diagonalized into $D$ by an orthogonal matrix whose first row is $\q^T$, and so this would contradict the choice of $\q$ in Step 2 of Algorithm \ref{algo:diag2trid}, and the thesis follows.
\end{proof}

\subsection{Singular tridiagonal with specific null-vector}\label{ZeroSum}
Next, we consider the following modification of the previous problem. \\

{\emph{Given $N$ real values
$\lambda_1=0 < \lambda_2 < \cdots < \lambda_N$, find an unreduced, network
tridiagonal matrix $L$, whose spectrum is given by these values, and such that the eigenvector associated to the $0$-eigenvalue is aligned with $\smat{1\\1 \\ \svdots \\1}$.}} 

Our construction will produce a generally nonsymmetric tridiagonal matrix $L$ as in 
\eqref{GivenTrid} and we will see that --in general-- one cannot require that there is a 
symmetric tridiagonal matrix satisfying our requests.  
The technique we used to resolve this problem is encoded in the following Algorithm,
which will be justified in Theorem \ref{Terminate} below.

%
%

\medskip

\begin{algorithm}[H]
	\caption{{\tt TridZeroRowSum}}\label{algo:zerorowsum}
	\begin{algorithmic}[1]
		\State Using Algorithm \ref{algo:diag2trid}, {\tt diag2trid}, generate an unreduced symmetric
		tridiagonal matrix $S$ as in  \eqref{SymmTrid} with eigenvalues
		$\lambda_1=0 < \lambda_2 < \cdots < \lambda_N$.
		\State 	Modify the first $N-1$ rows of $S$ as follows. 
		\hfill\newline
		(a) Let $\alpha_2$ such that $a_1+\alpha_2 b_2=0$.\hfill\newline
		(b) For $k=3,\dots, N$, let $\alpha_k$: 
		$\alpha_k b_k+a_{k-1}+\frac{b_{k-1}}{\alpha_{k-1}}=0$.
		\State 	
		The desired $L$ is then:
			\begin{equation}\label{FinalA}
				L=\bmat{a_1 & \alpha_2 b_2 & 0 & 0 & \cdots & 0 \\
					b_2/\alpha_2 & a_2 & \alpha_3 b_3 & 0 & \cdots & 0 \\	
					0 & b_3/\alpha_3 & a_3  & \alpha_4 b_4 & \ddots & 0 \\
					\vdots &  \ddots & \ddots & \ddots & \ddots & \vdots \\
					0 & \cdots & 0  & b_{N-1}/\alpha_{N-1} & a_{N-1} & \alpha_N b_N \\
					0 & \cdots & 0 & 0  & b_N/\alpha_N & a_N}\ .
			\end{equation}
	\end{algorithmic}
\end{algorithm}

\begin{thm}\label{Terminate}
	The above Algorithm \ref{algo:zerorowsum}, {\tt TridZeroRowSum}, is well defined and
	terminates with an unreduced tridiagonal matrix with
	spectrum given by $\{0,\lambda_2,\cdots, \lambda_n\}$, and eigenvector associated
	to the $0$-eigenvalue given (up to normalization) by
	$\e=\bmat{1\\ \vdots \\ 1}$.
\end{thm}
\begin{proof}
The algorithm is well defined as long as $\alpha_i \neq 0$ for all $i = 2, \ldots, N$. 
%
%
	Next, using Lemma \ref{BorderEvs} below,
	we show that the algorithm must complete the first $N-1$ steps,
	that is that, with $b_1=0$, for $k=2,\dots,N-1$ we have:
	\begin{equation*}
	\text{there exists, unique, } \alpha_k \ne 0 \text{ such that } \alpha_k b_k+a_{k-1}+\frac{b_{k-1}}{\alpha_{k-1}}=0, \\
	\end{equation*}
	or, equivalently,
	\begin{equation*}
	a_{k-1}+\frac{b_{k-1}}{\alpha_{k-1}}\ne 0. 
	\end{equation*}
	
Now, for $k=2$, we need to show that $a_1\ne 0$.  But this follows from 
		Lemma \ref{ai_positive}.

For $k>2$, we need to show that 
		$$a_{k-1}+\frac{b_{k-1}}{\alpha_{k-1}}\ne 0 . $$
		By contradiction, suppose that 
		$a_{k-1}+\frac{b_{k-1}}{\alpha_{k-1}}= 0$ and consider the principal minor
		of order $k-1$ of the matrix $B_{k-1}$ obtained during the reduction process, that is
		$$B_{k-1}=\bmat{a_1 & \alpha_2 b_2 & 0 & 0 & \cdots & 0 \\
			b_2/\alpha_2 & a_2 & \alpha_3 b_3 & 0 & \cdots & 0 \\	
			0 & b_3/\alpha_3 & a_3  & \alpha_4 b_4 & \ddots & 0 \\
			\vdots &   & \ddots & \ddots & \ddots & \vdots \\
			0 & \cdots & 0 & & b_{k-1}/\alpha_{k-1} & a_{k-1}}\ .$$
		But then this matrix is singular, since we would have that 
		$$B_{k-1}\bmat{1 \\ 1\\ \vdots \\1}=0$$
		and in particular $0$ would be an eigenvalue of $B_{k-1}$.  But, 
		the eigenvalues of $B_{k-1}$ are the same as the eigenvalues of the principal minor
		of order $k-1$ of $S$, that is of
		$$S_{k-1}=\bmat{a_1 & b_2 & 0 & 0 & \cdots & 0 \\
			b_2 & a_2 & b_3 & 0 & \cdots & 0 \\	
			0 & b_3 & a_3  & b_4 & \ddots & 0 \\
			\vdots &   & \ddots & \ddots & \ddots & \vdots \\
			0 & \cdots & 0 & & b_{k-1} & a_{k-1}}\ ,$$
		and so we would have found that a principal minor of $S$ is singular,
		contradicting Lemma \ref{BorderEvs}.
		
		Finally, since 
			by construction the first $N-1$ rows of $L$ have $0$-sum and are linearly independent, then the last row must also have $0$-sum (since it is linearly dependent on the previous $N-1$, given that $0$ is an eigenvalue).  	
			
		Therefore, the algorithm produces the matrix $L$ in \eqref{FinalA}, which clearly has the desired spectrum and rows that sum to zero. This concludes the proof.				
\end{proof}

\begin{lemma}\label{Unred}
Let $S\in \R^{N\times N}$ be a symmetric matrix with a $0$ eigenvalue.  Assume that $0$ is a simple
eigenvalue and let $\v$ be an associated eigenvector of length $1$.  
If no component of $\v$ is $0$, then $S$ is irreducible.  Conversely,
if $S$ is unreduced and tridiagonal, hence irreducible, 
with eigenvalues $\lambda_1=0<\lambda_2< \dots < \lambda_N$, then $v_i\ne 0$, for all $i=1,\dots, N$.
\end{lemma}
\begin{proof}
We show that, if no component of $\v$ is $0$, then $S$ is irreducible, by showing
that if $S$ is reducible, then there exists some $i$: $v_i=0$.
Indeed, if $S$ is reducible, then for some permutation $P$ we have $PSP^T=\bmat{S_1 & 0 \\ 0 & S_2}$, with  $S_1=S_1^T\in \R^{n_1\times n_1}$ and $S_2=S_2^T\in \R^{n_2\times n_2}$, $n_1,n_2\ge 1$ and $n_1+n_2=N$. So, we have $PSP^TP\v=0$, or $\bmat{S_1 & 0 \\ 0 & S_2}\bmat{\w_1\\ \w_2}=0$  with $\bmat{\w_1\\ \w_2}=P\v$.
Then, since the kernel of $S$ is 1-dimensional, we must have either $\w_1=0$ or $\w_2=0$,
giving the claim.

To show the converse statement, for $S$ unreduced and tridiagonal, 
suppose that $v_k=0$ for some $1\le k\le N$.  First, note that if $k=1$, 
then --
writing $S=\bmat{a_1 & \b^T \\ \b & B}$--  since $S\bmat{0\\ \v_2}=0$, 
$\v_2\in \R^{N-1}\ne 0$, then $B\v_2=0$ and this means that
$B$ has a $0$ eigenvalue, which contradicts Lemma \ref{BorderEvs} below.
Next, suppose $v_k=0$ and $k>1$.  Partition $S$ as follows:
$$S=\bmat{T_1 & \bmat{b_{k}\e_{k-1} & 0} \\ \bmat{b_{k}\e_{k-1}^T \\0} & T_2}$$
where $T_1$ is tridiagonal, unreduced,
of size $(k-1, k-1)$, and $T_2$ is tridiagonal, unreduced, of size $(N-k+1, N-k+1)$.
Writing $\v=\bmat{\v_1 \\0 \\ \v_2}$, $\v_1\in \R^{k-1}$, $\v_2\in \R^{N-k}$, then we 
must have $T_1\v_1=0$, and --because of Lemma \ref{BorderEvs}--
$T_1$ is invertbile
and this implies that $\v_1=0$.  Thus, we also have 
$T_2\bmat{0\\ \v_2}=0$ and thus, either $\v_2=0$ if $T_2$ is
invertible, contradicting that $\v\ne 0$, or $T_2$ is singular, contradicting
Lemma \ref{BorderEvs}.
\end{proof}

\begin{lemma}\label{BorderEvs}
	Given a symmetric, unreduced, tridiagonal matrix $S$ as in \eqref{SymmTrid}, with $N>1$, and
	with eigenvalues $\lambda_1=0<\lambda_2< \dots < \lambda_N$.  Then, any 
	leading (respectively, trailing) principal submatrix of $S$ of size $(p,p)$, $1\le p\le N-1$, is positive definite.
 \footnote{A leading (respectively, trailing) principal submatrix of size
	$(N-p,N-p)$, $p=1,\dots, N-1$, is the matrix obtained by deleting the bottom
	(respectively, top) $p$ rows and columns of $S$.}
\end{lemma}
\begin{proof}
	The proof follows from a refinement of classic results on interlacing of 
	eigenvalues for
	symmetric tridiagonal matrices.	In particular, the following result holds
	(see \cite[Problem 4.3.P17]{HJ-MatAna}):\\
	\smallskip\noindent
	{\emph{``Given $M=\bmat{C & \c \\ \c^T & d}$, symmetric, tridiagonal and unreduced, with $C\in \R^{N-1,N-1}$.
		Let $\mu_1 < \mu_2 < \dots < \mu_N$ be the eigenvalues of $M$, and let
		$\nu_1<\nu_2 < \dots < \nu_{N-1}$ be the eigenvalues of $C$.  Then, the $\nu_j$'s
		interlace properly the $\mu_i$'s.  That is, we have
		$$\mu_1<\nu_1 <\mu_2 <\nu_2 < \dots \mu_{N-1} < \nu_{N-1}< \mu_N.$$
	The result also holds if we partition $M=\bmat{a & \b^T \\ \b & C}$. ''}}
	
	We show the result for the leading principal submatrices.
	Let $S_k$, $k=1,2\dots, N-1$, be the principal submatrices of order $k$ of $S$,
	and let $\lambda_1^{(k)}<\dots <\lambda_k^{(k)}$ be their eigenvalues.  
	Using the proper interlacing result quoted above, in particular we must have:
	$$\lambda_1^{(N)}=\lambda_1=0<\lambda_1^{(N-1)}<\lambda_1^{(N-2)}<\dots <
	\lambda_1^{(1)}$$
	and the result follows.  The case of trailing principal submatrices is identical.
\end{proof}

We will also need the following result that refines Lemma \ref{Unred} relative to the eigenvector of $S$ associated to the $0$ eigenvalue.

\begin{lemma}\label{Ratios}
Let $S$ be a symmetric, unreduced, tridiagonal matrix as in
\eqref{SymmTrid}, with $N>1$, and with eigenvalues $\lambda_1=0<\lambda_2< \dots < \lambda_N$.  Let $\v$ be a unit eigenvector of $S$ associated to the $0$-eigenvalue.  Then, the entries of $\v$ all have the same sign. 
\end{lemma}
\begin{proof}
We are going to use a beautiful relation between the entries of the eigenvector and the Sturm sequence \eqref{Sturm}. Using \cite[Formula (15)]{Molinari}, it holds that 
$$v_k = c \ \frac{p_{k-1}(0)}{b_2\dots b_k},\,\ k=2,\dots, N,$$
where $c$ is a nonzero constant fixing $v_1$.
Because of Lemma \ref{Unred}, we know that all entries of $\v$ are not $0$ and thus we can write
$$\frac{v_{k+1}}{v_k}=\frac{1}{b_{k+1}} \frac{p_k(0)}{p_{k-1}(0)}\ .$$
In this last expression, both fractions in the right-hand-side are negative values.  In fact, for the first fraction this is obvious, since $b_{k+1}<0$; the second fraction is the ratio between the characteristic polynomials of the leading principal minors of order $k$ and $k-1$, evaluated at $0$.  But, because of the proper interlacing result on the eigenvalues of the principal minors (see the proof of Lemma \ref{BorderEvs}), the polynomials $p_k$ and $p_{k-1}$ assume opposite values at the origin, and so the ratio $p_k(0)/p_{k-1}(0)<0$ and the result follows.
\end{proof}

Finally, we conclude this section with the following result that summarizes the fact that with our construction we obtain a network tridiagonal matrix satisfying the structural form of Definition \ref{DefTrid}, which is what we wanted to achieve.  

\begin{thm}\label{NetworkL}
The matrix $L$ in \eqref{FinalA} satisfies the structural
assumptions of Definition \ref{DefTrid}; in particular, all
values $a_i$, $i=1,\dots, N$, in \eqref{SymmTrid} are strictly positive and the off diagonal entries $\alpha_i b_i$ (and of course $b_i/\alpha_i)$, $i=2,\dots, N$, are strictly negative.
\end{thm}
\begin{proof}
By looking at $L$ in \eqref{FinalA}, we observe that the $a_i$'s are the $a_i$'s of $S$ produced by Algorithm \ref{algo:diag2trid}, hence they are strictly positive because of Lemma \ref{ai_positive}.  Also, the values of $b_i$, $i=2,\dots, N$, are negative because they come from $S$.  So, we now show that the $\alpha_i$'s are
positive and the result will follow, since (by construction) the sum of the entries in each row is $0$.

Let $\v$ be the eigenvector of $S$ associated to the $0$-eigenvalue.  As in the proof of Theorem \ref{Terminate}, and because the $\alpha_j$'s are not zero in light of Lemma \ref{Unred}, we have the following relation between the entries of $\v$ and the $\alpha_j$'s, for some nonzero value of $c$:
$$c=v_1,\,\ c\alpha_2=v_2,\,\ c\alpha_2\alpha_3=v_3,\,
\dots ,\, c\alpha_2\cdots \alpha_N=v_N\ .$$
Therefore, we have that 
$$\alpha_k=\frac{v_k}{v_{k-1}}\ ,\,\ k=2,\dots, N ,$$
and using Lemma \ref{Ratios} the result follows.
\end{proof}

\subsection{Impact of symmetry}
In Example \ref{NoSymmNoCry}, we point out that, in general, one cannot also require that the sought tridiagonal matrix be symmetric.  

\begin{example}\label{NoSymmNoCry}
Suppose that, given $\lambda_1=0<\lambda_2<\lambda_3$, there exists a 
$3\times 3$ real symmetric unreduced tridiagonal matrix $T$ such that
(i) $T$ has eigenvalues $\{0, \lambda_2, \lambda_3\}$, and
(ii) the kernel of $T$ is spanned by $\bmat{1 \\ 1 \\ 1}$.
To satisfy condition (ii), $T$ must have the form
$T=\bmat{x & -x & 0 \\ -x & x+y & -y \\ 0 & -y & y }$, and 
to satisfy also condition (i), $x$ and $y$ must satisfy
$\begin{cases}
		2(x+y)=\lambda_2+\lambda_3 \\
		3xy=\lambda_2\lambda_3
\end{cases}$.
Solving with respect to $x$ and $y$ yields (uniquely) two pairs 
of solutions:
\begin{equation*}
	\begin{cases}
		x = \frac{1}{12}\left(\pm\sqrt{3} \sqrt{3 \lambda_2^2 - 10 \lambda_2\lambda_3 
			+ 3 \lambda_3^2} + 3 \lambda_2 + 3 \lambda_3\right) \\[.2cm]
		y = \frac{1}{12}\left(\mp\sqrt{3} \sqrt{3 \lambda_2^2 - 10 \lambda_2\lambda_3
			+ 3 \lambda_3^2} + 3 \lambda_2 + 3 \lambda_3\right) .
	\end{cases}
\end{equation*}
But, for $x$ and $y$ to be real valued, we must have
\begin{equation*}\label{eq:3by3ineq}
	10 \lambda_2\lambda_3 \le 3 \lambda_2^2 + 3 \lambda_3^2
\end{equation*}
In conclusion, there exists a real matrix 
$T=\bmat{x & -x & 0 \\ -x & x+y & -y \\ 0 & -y & y }$ having eigenvalues 
$0<\lambda_2<\lambda_3$ if and only if 
$10 \lambda_2\lambda_3 \le 3 \lambda_2^2 + 3 \lambda_3^2$, which is not
necessarily satisfied.
\end{example}

\section{Symmetric Tridiagonal Networks and the MSF}\label{TridMSF}
In this section we give some new results on the ratio between the first positive and the largest eigenvalue of symmetric positive semi-definite network tridiagonal matrices and further highlight the impact that these results have on synchronizability of symmetric tridiagonal networks. 

For every $N\in\N$, with $N\ge2$, we let $\cS$ denote the set of real $\NxN$ matrices as in \eqref{SymmTrid} that further are network tridiagonal matrices (that is, for which the entries on each row add up to zero) and that are positive semi-definite.  Further, we indicate with $T$ the $\NxN$ matrix in \eqref{TridStandard} (obviously, an element of $\cS$).
	

As usual, we let the eigenvalues be ordered in increasing fashion.  Namely, for given $S\in\cS$, we denote its eigenvalues by
\begin{equation*}
	\lam_1=0<\lam_2<\ldots<\lam_N \quad \text{and for} \quad T \quad \text{we call them} \quad	\mu_1=0<\mu_2< \ldots < \mu_N\ .
\end{equation*}

\begin{rem}\label{LinCombForS}
We observe that any $G\in\cS$ is a linear combination of elementary matrices  $E_1,\ldots,E_{N-1}$:
\begin{equation}\label{eq:linCombEk}
	G=\sum_{k=1}^{N-1}x_kE_k, \text{ for some } x_1,\ldots,x_{N-1}\in\R ,
\end{equation}
where for any $k=1,\ldots,N-1$, $E_k$ is the $N\times N$ matrix defined by
\begin{equation*}
	E_k(i,j)=
	\begin{cases}
		1 & \text{if } (i,j)=(k,k) \text{ or } (i,j)=(k+1,k+1),\\
		-1 & \text{if } (i,j)=(k,k+1) \text{ or } (i,j)=(k+1,k),\\
		0 & \text{otherwise}.
	\end{cases}
\end{equation*}
For instance, when $N=4$ we have
\begin{equation*}
	E_1=\bmat{
		1 & -1 & 0 & 0 \\
		-1 & 1 & 0 & 0 \\
		0 & 0 & 0 & 0 \\
		0 & 0 & 0 & 0 \\
	},\
	E_2=\bmat{
		0 & 0 & 0 & 0 \\
		0 & 1 & -1 & 0\\
		0 & -1 & 1 & 0 \\
		0 & 0 & 0 & 0 \\
	},\
	E_3=\bmat{
		0 & 0 & 0 & 0 \\
		0 & 0 & 0 & 0 \\
		0 & 0 & 1 & -1\\
		0 & 0 & -1 & 1\\
	}.
\end{equation*}
\end{rem}

The following result for parameter dependent symmetric matrices with distinct eigenvalues is well known (e.g., see \cite{Tao})),
and it will be used in the proof of Theorem \ref{thm:diffCoupl} below.
\begin{lemma}[First and second Hadamard variation formulae]\label{lem:varFormulas} Let $A(t)$ be a real symmetric matrix valued function that is twice continuously differentiable in $t$, and such that $A(t)$ has simple eigenvalues for all $t$. Then, also its eigenvalues $\lam_k(t)=\lam_k(A(t))$ and corresponding orthonormal eigenvectors $\u_k(t)=\u_k(A(t))$ depend smoothly on $t$ and we have:
	\begin{equation}\label{eq:Hada1}
		\dot\lam_k(t)=\u_k^T(t)\dot A(t) \u_k(t),
	\end{equation}
	\begin{equation}\label{eq:Hada2}
		\ddot\lam_k(t)=\u_k^T(t)\ddot A(t) \u_k(t)+2\sum_{j\ne k}\frac{\abs{\u_j^T(t)\dot A(t)\u_k(t)}^2}{\lam_k(t)-\lam_j(t)}.
	\end{equation}
\end{lemma}

\begin{thm}\label{thm:diffCoupl} For all integers $N\ge3$, and with the notation above, we have:
	\begin{equation}\label{eq:diffCoupl}
		\max_{S\in \cS} \frac{\lam_2(S)}{\lam_N(S)} =\frac{\mu_2}{\mu_N}=\frac{\sin\left(\frac{\pi}{2N}\right)^2}{\sin\left(\frac{(N-1)\pi}{2N}\right)^2},
	\end{equation}
	and the maximum is attained only at scalar multiples of $T$. 
\end{thm}
\begin{proof}
With notation from \eqref{TridStandard}, we have
\begin{equation*}
	T\v_k=\mu_k\v_k, \text{ for all } k=1,\ldots,N\ .
\end{equation*}
Given any $S\in\cS$ that is not a scalar multiple of $T$, define the following matrix valued function
\begin{equation*}
	A(t)=(1-t)S+tT,\ t=[0,1],
\end{equation*}
and the following real valued function, which is well defined and strictly positive for all $t\in[0,1]$:
\begin{equation}
	r(t)=\frac{\lambda_2(A(t))}{\lambda_N(A(t))}.
\end{equation}
We are going to show that:
\begin{enumerate}[(i)]
	\item $\dot r(1)=0$;
	\item $r$ is non-decreasing for all $t\in[0,1]$;
	\item $r$ is strictly increasing for all $t\in[1-\delta,1]$, for some $0<\delta\le1$. 
\end{enumerate}
Then the conclusion of the theorem will follow from (i)-(iii). In particular, (iii) shows that the maximum in \eqref{eq:diffCoupl} is attained only at scalar multiples of $T$.

\noindent{\bf Proof of point (i).}
Consider the map
\begin{equation}
	f:\x=\smat{x_1\\x_2\\ \svdots \\x_{N-1}}\in\R^{N-1}\mapsto\frac{\lambda_2(T+G(\x))}{\lambda_N(T+G(\x))},
\end{equation}
where $G=G(\x)$ is defined as in \eqref{eq:linCombEk}. We want to show that $\nabla f=0$ at $\x=0$.
Using \eqref{eq:Hada1}, we obtain
\begin{equation}
\pder{f}{x_j}\Big\rvert_{\x=0}=\frac{ \mu_N \v_2^TE_j \v_2-\mu_2\v_N^TE_j \v_N}{\mu_N^2}.
\end{equation}
Hence, we have
\begin{equation}\label{eq:iff_zeroGradient}
	\pder{f}{x_j}\Big\rvert_{\x=0}=0\quad\text{ if and only if }\quad\frac{\v_2^T E_j \v_2}{\v_N^T E_j \v_N}=\frac{\mu_2}{\mu_N}.
\end{equation}
Direct computation yields
\begin{equation}\label{eq:derRatio}
	\frac{\v_2^T E_j \v_2}{\v_N^T E_j \v_N}=
	\frac{\left(v_2(j)-v_2(j+1)\right)^2}
	{\left(v_N(j)-v_N(j+1)\right)^2}.
\end{equation}
Substituting \eqref{DiffCoupSpec}
into the right-hand side of \eqref{eq:derRatio} and applying trigonometric identities, we obtain
	\begin{equation*}
		\begin{split}
&			\frac{\big(v_2(j)-v_2(j+1)\big)^2}{\big(v_N(j)-v_N(j+1)\big)^2} =
			\frac{\left(\cos\dfrac{(2j-1)\pi}{2N}-\cos\dfrac{(2j+1)\pi}{2N}\right)^2}
			{\left(\cos\dfrac{(N-1)(2j-1)\pi}{2N}-\cos\dfrac{(N-1)(2j+1)\pi}{2N}\right)^2} = \\
			& = 
			\frac{\left(\sin\dfrac{j\pi}{N}\sin\dfrac{\pi}{2N}\right)^2}{\left(\sin\dfrac{(N-1)j\pi}{N}\sin\dfrac{(N-1)\pi}{2N}\right)^2}
=\frac{\left(\sin\dfrac{\pi}{2N}\right)^2}{\left(\sin\dfrac{(N-1)\pi}{2N}\right)^2}=\frac{\mu_2}{\mu_N}.
		\end{split}
	\end{equation*}
	This proves the statement on the right of \eqref{eq:iff_zeroGradient}, and hence concludes the proof of point (i).\smallskip
	
\noindent{\bf Proof of point (ii).}
Recall that $A(t)$ is linear in $t$, and that $0=\lam_1(t)<\lam_2(t)<\ldots<\lam_N(t)$ for all $t\in(0,1]$. Then, it follows from \eqref{eq:Hada2} that $\lam_N(t)$ and $\lam_2(t)$ are, respectively, convex and concave functions of $t$ in $[0,1]$. Differentiating $r$, we obtain
\begin{equation}
	\dot r= \frac{\dot\lam_2\lam_N-\dot\lam_N\lam_2}{\lam_N^2}.
\end{equation}
Define $\varphi(t):=\dot\lam_2(t)\lam_N(t)-\dot\lam_N(t)\lam_2(t)$ for all $t\in[0,1]$. We have
	\begin{equation}
		\dot\varphi=\ddot\lam_2\lam_N-\ddot\lam_N\lam_2\le 0.
	\end{equation}
	Since $\dot r(1)=0$, we have $\varphi(1)=0$, and therefore $\varphi(t)\ge0$ for all $t\in[0,1]$. This proves $\dot r(t)\ge 0$ for all $t\in[0,1]$, i.e. $r$ is non-decreasing on $[0,1]$. 
	
	\smallskip\noindent{\bf Proof of point (iii).}
	We are going to show that $\ddot\lam_2(t)<0$ in $(1-\delta,1)$ for some $\delta>0$. This implies that $r$ is strictly decreasing in $(1-\delta,1)$, and hence $r(0)>r(1)$.
	
We know that $\ddot\lam_2(1)\le0$. Suppose $\ddot\lam_2(1)=0$.  Because of \eqref{eq:Hada2}, we have that $\v_j^T(T-S)\v_2=0$ for all $j\ne2$ (recall that $\v_1^T(S-T)$ is the zero vector). Then, we have that $(T-S)\v_2$ must be a scalar multiple of $\v_2$, possibly zero. In any case, this means that $\v_2$ is an eigenvector of $S$. A non-zero vector $\w=\smat{w_1\\ \svdots \\w_N}\in\R^N$ is an eigenvector of $S$ if and only if there exists $\alpha\in\R$ such that 
$(S-\alpha I)\w=0$. 
Setting $s_j=-S_{j,j+1}$ for all $j=1,\ldots,N-1$, we recast this as a linear system in $s_1, \ldots,s_{N-1}, \alpha$ as follows:
\begin{equation}\label{eq:Wmat_def}
	\bmat{
		w_1-w_2 & & & & & -w_1 \\
		-w_1+w_2 & w_2-w_3& & & & -w_2 \\
		& -w_2+w_3 & \ddots & & & \vdots\\
		& & \ddots& & w_{N-1}-w_N & -w_{N-1}\\
		& & & & -w_{N-1}+w_N& -w_N\\ 
	}
	\bmat{s_1 \\ s_2 \\ \vdots \\ s_{N-1} \\ \alpha}=
	\bmat{0 \\ 0 \\ \vdots \\ 0 \\ 0}.
\end{equation}
Performing row-reduction on \eqref{eq:Wmat_def},
we obtain the equivalent system:
\begin{equation}\label{eq:Wmat_transformed}
	\bmat{
		w_1-w_2 & & & & -w_1 \\
		& w_2-w_3& & & -w_1-w_2 \\
		& &  \ddots & & \vdots\\
		& & & & -w_1-w_2-\ldots-w_{N}
	}
	\bmat{s_1 \\ s_2 \\ \vdots \\ s_{N-1} \\ \alpha}=
	\bmat{0 \\ 0 \\ \vdots \\ 0 \\ 0}.
\end{equation}
Let us denote the matrix in \eqref{eq:Wmat_transformed} by $W$ and set $\w = \v_2$. Then, since $w_j \ne w_{j+1}$ for all $j = 1, \ldots, N-1$ (see \eqref{DiffCoupSpec}), and $\v_1^T \v_2 = 0$, it follows that $\rank(W) = N-1$. Recall that a non-trivial solution to \eqref{eq:Wmat_transformed} is known and is given by $s_1=\ldots=s_{N-1}=1$ and $\alpha=\mu_2$. It follows that $S$ must be a scalar multiple of $T$, which is something we have excluded. Therefore, we must have $\ddot\lam_2(1)<0$.  Being $\ddot\lam_2$ continuous, then it must remain strictly negative in a (left) neighborhood of $1$.  This concludes the proof.
\end{proof}

\begin{coro}\label{cor:diffCoupl} Suppose we are given $0=\nu_1<\nu_2<\ldots<\nu_N$, with
	\begin{equation}
		\frac{\nu_2}{\nu_N}>\frac{\sin\left(\frac{\pi}{2N}\right)^2}{\sin\left(\frac{(n-1)\pi}{2N}\right)^2} \ .
	\end{equation}
Then, there is no matrix $S\in\cS$ such that $\lambda_j(S)=\nu_j$, $j=1,\ldots,N$.
\end{coro}

\begin{rem}\label{KeyRemark}
For our purposes, the key consequence of Corollary \ref{cor:diffCoupl} is the following.
Suppose that there is a finite interval of values $\eta$ for which the MSF is negative.  Then, because of \eqref{l2lN}, if
\begin{equation*}
	\frac{\min\{\eta: \text{the MSF is negative}\}}{\max\{\eta: \text{the MSF is negative}\}}>\frac{\sin\left(\frac{\pi}{2N}\right)^2}{\sin\left(\frac{(N-1)\pi}{2N}\right)^2},
\end{equation*}
it is not possible to obtain a negative value of the MSF with {\bf any} symmetric tridiagonal network Laplacian.
\end{rem}


\section{Numerical Results}\label{NumRes}
Here we show how our technique works in practice. We give two examples, one is a
network of van der Pol oscillators with periodic synchronous orbit, the other is a 
network of R\"ossler oscillators with chaotic synchronous orbit.  Among our goals
is to show that, in general, the symmetric tridiagonal structure
\eqref{TridStandard} may fail to give stability of the synchronous orbit (that is, it won't give a  negative value of the MSF, no matter how large is $\sigma$ in \eqref{TridStandard}), but our technique is always able to give a negative value of the MSF, if there is an $\eta$-interval where the MSF is negative.


For both our examples, we proceed as follows. 
\begin{itemize}
	\item[1.] We consider the parametrized linear system \eqref{ParamLinear}
	and compute the MSF in function of $\eta$. 
	\item[2.] If there exist any, we select an 
	interval $[\eta_1, \eta_2]$ so that for 
	$\eta \in [\eta_1, \eta_2]$, the MSF is negative. 
	\item[3.] We use Algorithm {\tt TridZeroRowSum} to build the Laplacian $L$ as in 
	\eqref{FinalA} with nonzero eigenvalues in $[\eta_1, \eta_2]$, and null vector 
	 $\frac{1}{\sqrt{N}}\smat{1\\ \svdots \\ 1}$.
\end{itemize}

\subsection{Van der Pol}\label{vdP}
We consider a network of $N$ identical Van der Pol oscillators. The single agent satisfies the following equation
\begin{equation}\label{VdP_eq}
	\left \{ \begin{aligned}
		\dot y_1= & y_2 \\
		\dot y_2= &  -y_1+y_2(1-y_1^2)
	\end{aligned} \right . 
\end{equation}
and we choose the following coupling matrix 
$E=\begin{pmatrix} 0 & 0 \\ 1 & 0 \end{pmatrix}$. 
In Figure \ref{VdP_fig} on the left we plot the MSF for $\eta \in [0,\,\ 0.5]$. The MSF is negative for $\eta>0.39$ and remains negative. at least for $\eta  \leq 50$.    
  If we couple the agents via symmetric diffusive coupling as in Example \ref{DiffCouple}, then in order to synchronize the network we must impose $\sigma \lambda_2>0.39$ with $\sigma$ constant coupling strength. We then need large values of the coupling strength, namely $\sigma>40.5$ for $N=32$. 
  We instead employ our technique and consider an asymmetric tridiagonal coupling. 
  
We have some freedom on what values we select for the $N-1$ 
eigenvalues of $L$.  We experimented with many different
choices for these values and below we report on two different experiments: 
i) select the eigenvalues linearly spaced in $[1, 10]$, and
ii) take the eigenvalues to be Chebyshev points of the first kind (rescaled
to $[1, 10]$ or to larger intervals, as needed). 
For $N=32$, all the elements of the coupling matrix are less than $6.2$. We take an initial condition on the attractor and perturb it with a normally distributed perturbation vector. Then we integrate the network with a 4-th order Runge Kutta method and fixed stepsize $h$. For these methods, theoretical results guarantee that for $h$ sufficiently small the numerical method has a closed invariant curve and that the distance of this curve from the synchronous periodic orbit is $O(h^4)$ (see \cite{Eirola}).  \\

We synchronize $32$ agents with both choices of eigenvalues. To witness, in Figure \ref{VdP_fig} on the right we plot $\max_{i=1, \ldots, 31} \|\x_i(t)-\x_{i+1}(t)\|_2$ at the grid points. The greatest value of the MSF is obtained for $\eta=1$, and it is $\simeq -0.13$. We plot $e^{-0.13 t}$ as well in order to appreciate the convergence speed to the synchronous solution. The plots are obtained for one initial condition, but the behavior we observe is consistent for every normally distributed random perturbation we considered. The transient in this case is relatively short. \\

\begin{figure}
	\begin{minipage}[c]{0.49\linewidth}
		\centering
		\includegraphics[width=\linewidth]{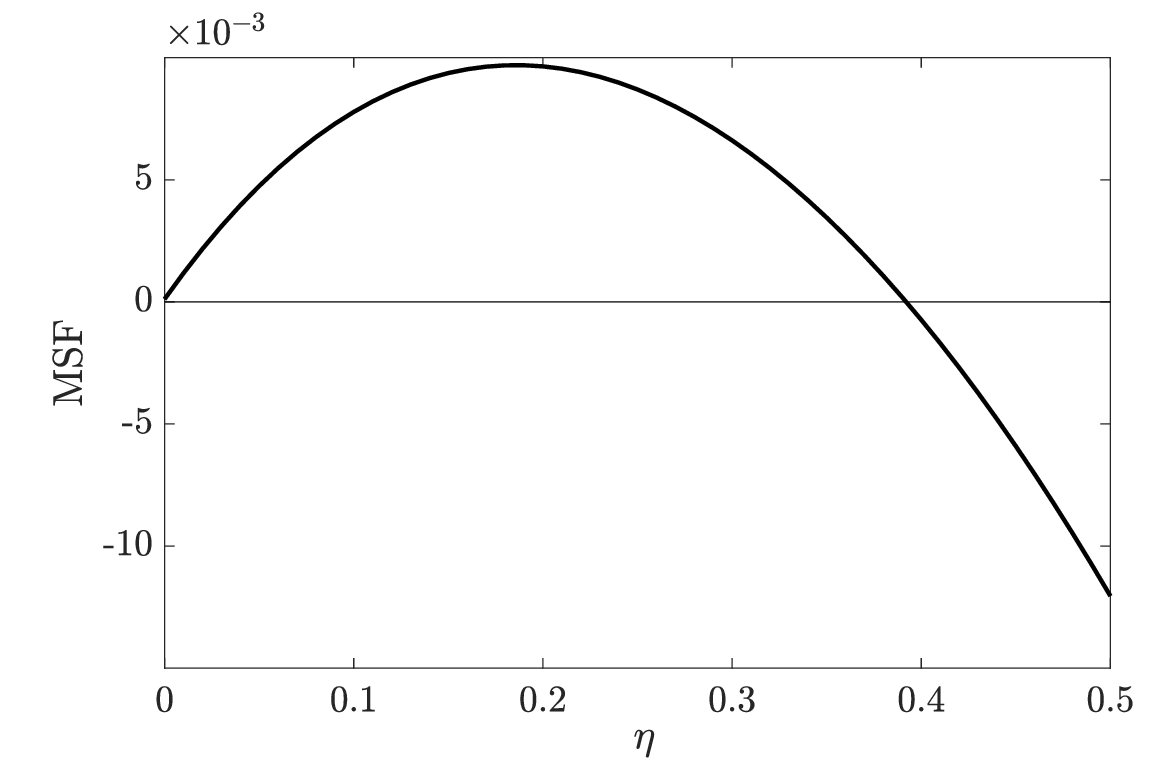}
	\end{minipage}
	\begin{minipage}[c]{0.49\linewidth}
		\centering
		\includegraphics[width=\linewidth]{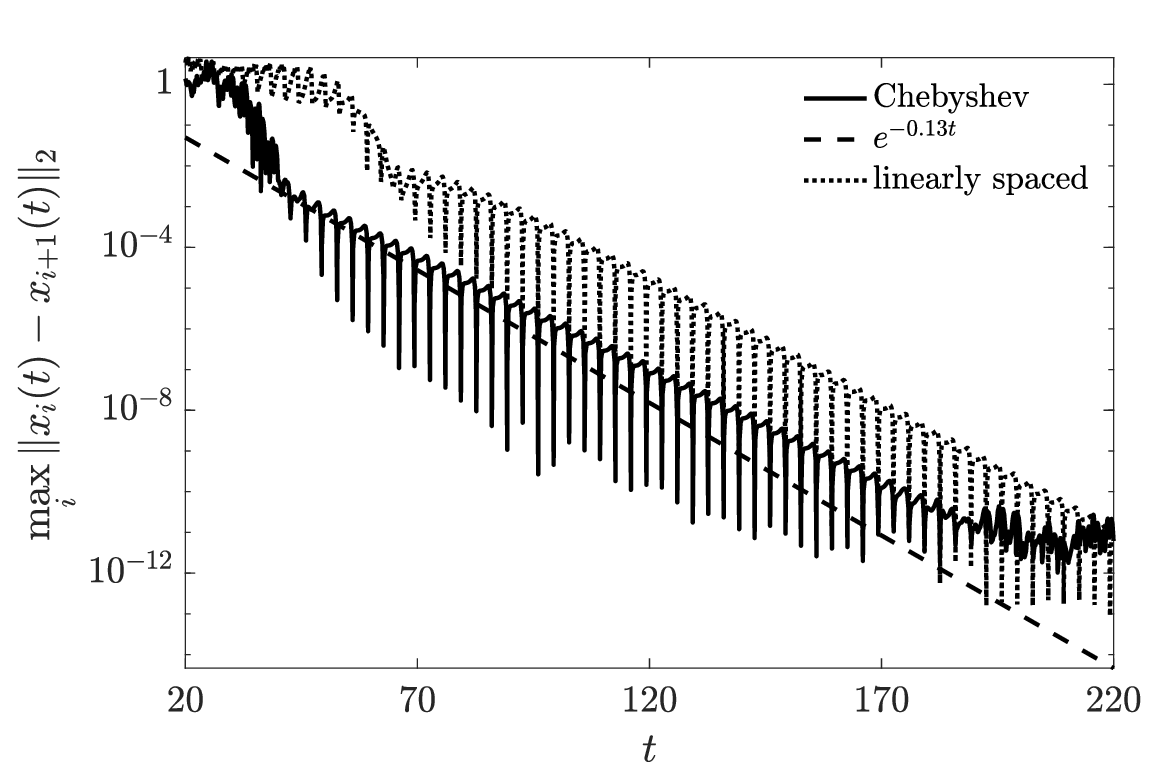}		
	\end{minipage}
\caption{Van der Pol, N=32.  Left: MSF. Right: 2-norm of the difference between agents.}
\label{VdP_fig} 
\end{figure}

\subsection{R\"ossler}\label{Rossler}
Next, we consider a network of $N$ identical R\"ossler oscillators. Each agent
satisfies the system
\begin{equation}\label{Ross_eq}
	\left \{ \begin{aligned}
		\dot y_1= & -y_2-y_3 \\
		\dot y_2= &  y_1+0.2y_2 \\
		\dot y_3= & 0.2+(y_1-9)y_3 
	\end{aligned} \right .
\end{equation}
and we choose the coupling matrix $E \in \R^{3 \times 3}$ in \eqref{ManyAgents2}
given by $E=\begin{pmatrix} 1 & 0 & 0 \\ 0 & 0 & 0 \\ 0 & 0 & 0 \end{pmatrix}$.
In Figure \ref{Rossler1_fig}, on the left, we plot the MSF in function of $\eta$. 
The MSF is negative in the interval $\eta\in [0.19 \,\ 4.61]$. 

If we use diffusive coupling with a 
constant coupling strength $\sigma$, like in Example \ref{DiffCouple},
then (see the explicit values of the eigenvalues given in Example \ref{DiffCouple}),
in order to synchronize $N$ agents, we 
need $\sigma \lambda_2>0.19$ and $\sigma \lambda_N<4.61$.
A necessary condition for $N$ agents to synchronize is then, see \eqref{l2lN},  $\frac{\lambda_N}{\lambda_2}<\frac{4.61}{0.19}$ and as soon as 
$N>7$ this condition is not satisfied and hence we cannot hope to 
synchronize more than $7$ agents with diffusive coupling and constant coupling 
strength.
In \cite{PecoraCarroll} the authors use diffusive coupling in a circular array and can 
synchronize up to $10$ agents, but already with $16$ agents the necessary conditions 
for synchronization are not met anymore.\footnote{The authors of 
\cite{PecoraCarroll} point our that, with symmetric tridiagonal Laplacians there 
will always be an upper limit in the size of a the network in order to obtain a 
stable synchronous chaotic orbit. 
With our technique, in principle there is not such limitation.} 

In what follows we couple and synchronize a network of 64 R\"ossler agents with
our technique, about the chaotic orbit of a single R\"ossler oscillator.
Of course, as we observed in our numerical experiments, the lack of symmetry causes a large transient and the basin of attraction of the synchronous periodic orbit is in general affected by it. 

We report on experiments with two different sets of eigenvalues: 
linearly spaced in $[0.5, 3]$, and Chebyshev points of the first kind (rescaled
to $[0.5, 3]$). 
We consider initial conditions obtained by adding a normally distributed perturbation 
with variance $1$ of a synchronous initial condition and integrate \eqref{ManyAgents2}
to verify that indeed we obtain synchronization.
In Figure \ref{Rossler1_fig} we plot 
$\max_{i=1, \ldots, 63} \|\x_i(t)-\x_{i+1}(t)\|_2$ for the two choices of eigenvalues.
For $\eta=0.5$, the corresponding value of the MSF is $\simeq -0.15$ and the dashed 
line in the right plot is the graph of $e^{-0.15t}$. It is clear that, after a transient, the 
convergence speed of the two perturbed solutions is also $e^{-0.15t}$.  

\begin{figure}
	\begin{minipage}[c]{0.49\linewidth}
		\centering
		\includegraphics[width=8cm]{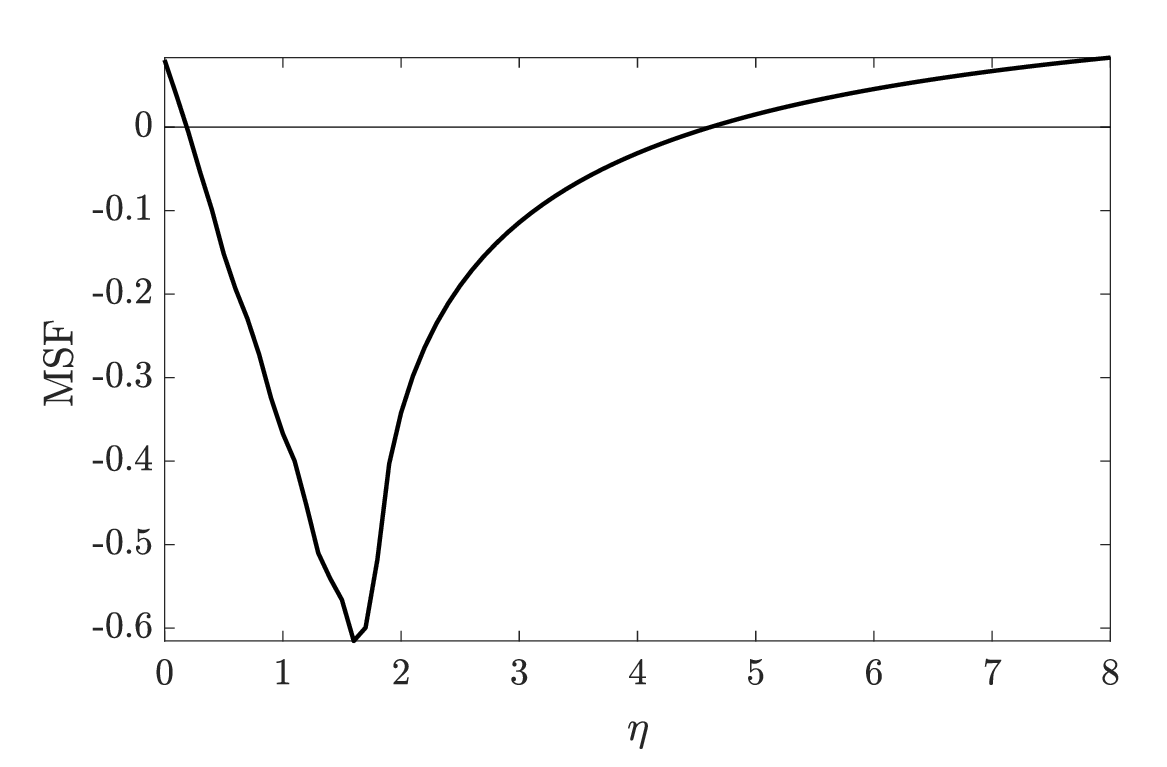}
	\end{minipage}
	\begin{minipage}[c]{0.49\linewidth}
		\centering
		\includegraphics[width=8cm]{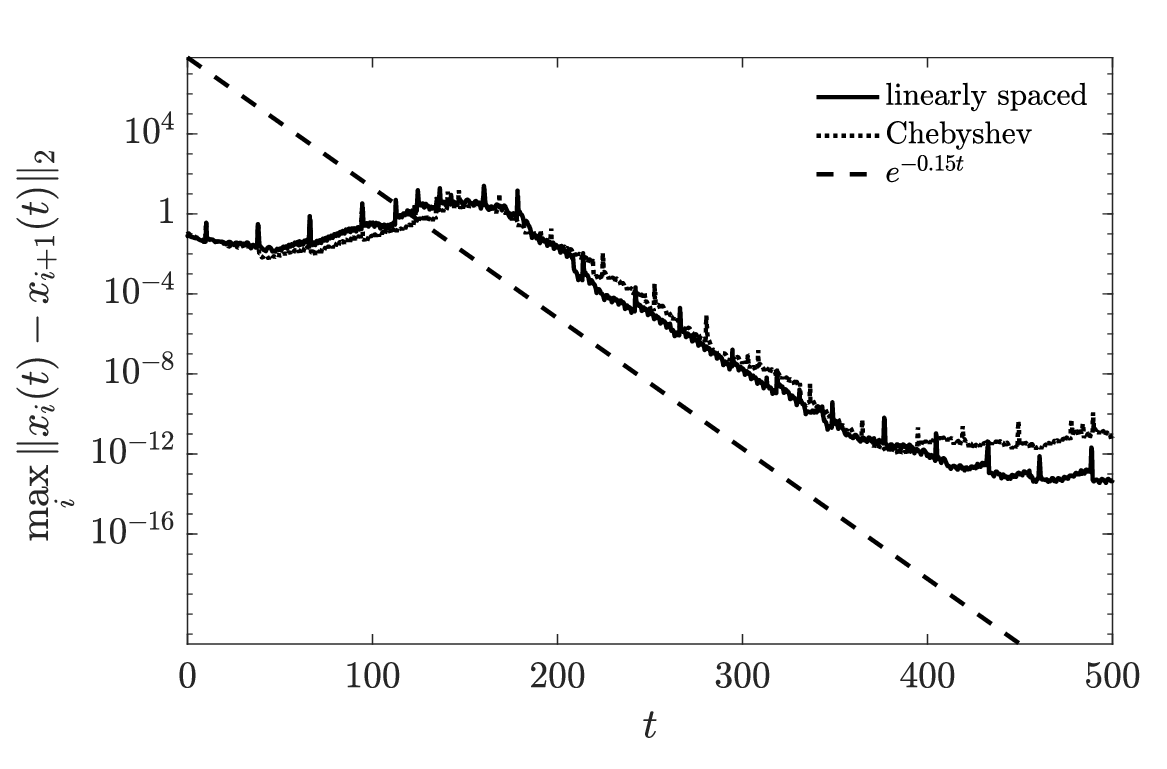}
	\end{minipage}
\caption{Left: MSF for R\"ossler. Right: 2-Norm of the difference between 
	agents.}\label{Rossler1_fig} 
\end{figure}

\section{Conclusions}\label{Concl}
In this work, we gave two contributions.  First, we built a tridiagonal matrix $L$ with eigenvalues $\lambda_1=0<\lambda_2<\cdots <\lambda_N$ and null-vector $\e=\smat{1\\ \svdots \\ 1}$.  Then, we used this result to achieve --if possible-- that a synchronous orbit of a tridiagonal network associated to the matrix $L$ above is asymptotically stable, in the sense of having an associated negative Master Stability Function (MSF).  Our matrix $L$, in general, is not symmetric, and we further gave some results highlighting the limitations occurring if we also require symmetry.
From the practical point of view, the main consequence of our results is that, if at all possible to have a negative MSF with tridiagonal network structure for the Laplacian, then  we can have a stable synchronous motion, and we gave numerical examples to show this.
At the same time, we also observed that, when we increase the number of agents, it is not always  possible to achieve synchronization with arbitrary precision.  We leave investigation of this fact to a future numerical study.

\end{document}